\documentclass[10pt,a4paper]{amsart}
\usepackage[utf8]{inputenc}
\usepackage[T1]{fontenc}
\usepackage{amsmath}
\usepackage{upquote}
\usepackage{bm}
\usepackage{amsfonts}
\usepackage{amssymb}
\usepackage[english]{babel}
\usepackage{tabularx}
\usepackage{mathtools}
\usepackage{adjustbox}
\usepackage{graphics}
\usepackage{hyperref}
\usepackage{xcolor}
\usepackage{diagbox}
\usepackage{color}
\usepackage{tikz}
\usepackage{chngcntr}
\usepackage{array}
\newcolumntype{H}{>{\setbox0=\hbox\bgroup}c<{\egroup}@{}}
\usetikzlibrary{calc}
\numberwithin{equation}{section}
\newtheorem{thm}{Theorem}[section]

\newtheorem{lm}[thm]{Lemma}
\newtheorem{cor}[thm]{Corollary}

\theoremstyle{remark}

\newtheorem{re}[thm]{Remark}
\theoremstyle{definition}

\newtheoremstyle{case}{}{}{}{}{}{:}{ }{}
\theoremstyle{case}

\counterwithin*{case}{thm}
\newtheoremstyle{caso}{}{}{}{}{}{:}{ }{}
\theoremstyle{caso}

\DeclareRobustCommand{\svdots}{
  \vbox{%
    \baselineskip=0.33333\normalbaselineskip
    \lineskiplimit=0pt
    \hbox{.}\hbox{.}\hbox{.}%
    \kern-0.2\baselineskip
  }%
}

\newcommand{\N}{\mathbb{N}}

\newcommand{\lattice}{\Lambda }
\theoremstyle{remark}

\allowdisplaybreaks

\makeatletter
\let\@@pmod\pmod
\DeclareRobustCommand{\pmod}{\@ifstar\@pmods\@@pmod}
\def\@pmods#1{\mkern4mu({\operator@font mod}\mkern 6mu#1)}
\makeatother

\title{Submultiplicative Polynomials in Combinatorics}
\author{Krystian Gajdzica, Bernhard Heim, Markus Neuhauser and Błażej Żmija}
\address{Theoretical Computer Science Department \\ Faculty of Mathematics and Computer Science\\ Jagiellonian University\\ Łojasiewicza 6\\ 30-348 Kraków\\ Poland}
\email{krystian.gajdzica@uj.edu.pl}

\address{Faculty of Mathematical and Natural Sciences\\
Mathematical Institute\\
University of Cologne\\
Weyertal 86–90\\
50931 Cologne\\
Germany
}
\email{bheim@uni-koeln.de}

\address{Kutaisi International University\\
5/7\\
Youth Avenue\\
Kutaisi\\
4600 Georgia
}
\email{}
\address{Lehrstuhl f\"ur Geometrie und Analysis\\
RWTH Aachen University\\
52056 Aachen\\
Germany
}
\email{markus.neuhauser@kiu.edu.ge}

\address{Faculty of Applied Mathematics, AGH University of Krakow, Al. Mickiewicza 30, 30-059 Kraków, Poland}
\email{bzmija@agh.edu.pl}


\begin{document}

\begin{abstract}
For normalized sequences
$\left(g(n)\right)_{n\in\mathbb{N}}$ we consider recursively defined polynomials $P_n^g(x)$. In this paper we study their submultiplicative property, viewed as a 
Bessenrodt--Ono type inequality for the partition function,
and provide an effective criterion for establishing it. 
\end{abstract}

\setlength{\parindent}{10mm}
\maketitle

\newpage

\section{Introduction}
We investigate submultiplicative sequences.
They appear frequently in chemistry, statistical physics,
and combinatorics. The Nobel laureate P. Flory \cite{Flory} was among the first to study self-avoiding walks in lattices to model the real-life behavior of chain-like entities such as solvents and polymers.
Let $c_n$ be the number of self-avoiding walks in a lattice
$\lattice $ 
visiting every vertex at most once and taking $n$ steps,
started from some fixed vertex, e.g., the origin (see, \cite{Duminil}).
Then 
\begin{equation}\label{def:submultiplicative}
c_{n} \, c_m \geq c_{m+n} \qquad \text{ for 
all } n,m \geq 1.
\end{equation}
This is
called the submultiplicative property of the sequence.
Due to Fekete's Lemma, the limit
if $n$ goes to infinity of $\sqrt[n]{c_n}$ exists and is called the connective constant $\mu
\left( \lattice \right)
$. Duminil-Copin and Smirnov have determined this value for the hexagonal lattice \cite{DS}. In many other cases the value of the connective constant is unknown.

In the context of partitions this property (\ref{def:submultiplicative}) was detected by Bessenrodt and Ono 
\cite{BO} for the partition function for $n,m \geq 5$.
Chern, Fu, and Tang \cite{Chern} generalized this result to $k$-colored partitions $p_{-k}(n)$ and showed that 
(\ref{def:submultiplicative}) is satisfied for all integers $k \geq 3$ and $n,m \geq 1$. Further, let $P_n^{\sigma}(x)$ be the 
D'Arcais polynomials. Then for all $x \geq 3$ and $n,m \geq 1$
these values are also submultiplicative \cite{GHN, HN6}.

In this paper we generalize D'Arcais polynomials
associated with
normalized arithmetic functions $g(n)$ and classify their implication on submultiplicative properties of $P_n^g(x)$, which are defined by
\begin{equation*}
P_n^g(x) := \frac{x}{n} \sum_{k=1}^n g(k) P_{n-k}^g(x)
\end{equation*}
with the initial condition $P_0^g(x)=1$. Let $\sigma_{r}(n):= \sum_{d \mid
n} d^r$. Then $\sigma(n)= \sigma_1(n)$ gives the D'Arcais polynomials, where $p_{-k}(n)= P_n^{\sigma}(k)$.
We state our main theorem, where the following inequality for a fixed $x_0 \in \mathbb{R}_{>0}$ provided by
\begin{equation}\tag{$\ast_n$}\label{eq:star-n}
P_n^g(x_0) \geq \max_{1\le k\le n}
\frac{g(n+k)}{g(k)}
\end{equation}
plays an important role.
\begin{thm}\label{thm:main}
Let $\ell\in\N$, and let $g_{\ell}(n)$ be a normalized sequence. 
Suppose that the inequalities
\begin{equation}\label{eq:gl-growth}
n^{\ell }\leq g_{\ell }\left( n\right) \leq n^{\ell +1}.
\end{equation}
hold for every $n\in\mathbb{N}$. Then $
\left(
P_{n}^{g_{\ell }}\left( x\right) \right)
_{n}
$ is submultiplicative for all $a,b\geq 1$ and $x\geq x_0$, where $x_0 = g_{\ell}(2)$ and \eqref{eq:star-n} is satisfied for all $1 \leq n \leq n_{\ell}$ with $n_1=8$, $n_2=3$, $n_3=2$  and $n_\ell =1$ otherwise.
\end{thm}
To cover also $k$-regular and $k$-colored partitions, we 
extend this result to the case $1 \leq g(n) \leq \sigma(n)$. 
\begin{thm}\label{thm:main 2}
Let the normalized sequence $\left(
g\left( n\right) \right)
_n$ satisfy $1 \leq g(n) \leq \sigma(n)$ for all $n \in \mathbb{N}$. Let further
\begin{equation}
3g(2)\left(g(2)+3\right)\geq2g(4)\hspace{5mm}\text{and}
\hspace{5mm}
g(3)\left(2g(3)+9g(2)+9\right)\geq2g(6).
\label{condition3}
\end{equation}
Then the
submultiplicative property  
\begin{equation}
P_n^g(x) \, P_m^g(x) \geq P_{n+m}^g(x)
\label{submultiplicative}
\end{equation}
holds true for all $m,n \geq 1$ and $x \geq 3$. Moreover
regardless of condition (\ref{condition3}) the
submultiplicative property (\ref{submultiplicative})
holds true for all $m,n\geq 1$ and $x\geq 4$.
\end{thm}

\begin{re}
Note that condition (\ref{condition3}) is automatically
fulfilled if
$g\left( 2\right)
\geq 2$ or $g\left( 4\right) \leq 6$
and $g\left( 6\right) \leq 10$. 
\end{re}

Therefore, let $g^{(k)}(n):= \sigma(n) - k \sigma(n/k)$. Here $\sigma(y)=0$ if $y \not\in \mathbb{N}$. Then $P_n^{g^{(k)}}(r)$ is equal to the number of all $k$-regular partitions with $r$-colors. This implies that the sequence of
$k$-regular and $k$-colored partitions is submultiplicative for all $k >1$ and $r \geq 3$.

Let us describe the content of this work in some details. Section 2 gains some fundamental properties related to the family of polynomials $\left(
P_n^g(x)\right)
_n$, which are used throughout the paper. In Section 3, we prove Theorem \ref{thm:main}. Section 4 exhibits some applications of  Theorem \ref{thm:main}. Finally, the last section is devoted to the proof of Theorem \ref{thm:main 2}.

\section{Preliminaries}

Throughout the paper we use the following conventions. By $\mathbb{R}$, $\mathbb{N}$,
and $\mathbb{N}_{\geq m}$ we mean the set of real numbers, the set of positive integers,
and the set of positive integers greater or equal than $m$, respectively.

Let us point out that by $(g(n))_{n\in\mathbb{N}}$, we always mean a sequence of positive real numbers such that $g(1)=1$. For a sequence $g(n)$ we assign a family of polynomials defined recursively by
\begin{equation*}
P_n^g(x) := \frac{x}{n} \sum_{k=1}^n g(k) P_{n-k}^g(x)
\end{equation*}
with initial condition $P_0^g(x):=1$. In such a setting, we automatically obtain that $P_n^g(x)$ is a polynomial of degree $n$ with the leading coefficient $1/n!$. Furthermore, one can derive
a recurrence formula for its derivative.

\begin{lm}\label{Lemma: Recurrence for P_n^g}
    Under our assumptions on $(g(n))_{n\in\mathbb{N}}$ and $P_n^g(x)$, we have that
    \begin{align*}
        \frac{\mathrm{d}
        }{\mathrm{d}x}P_{n}^{g}(x)=\sum_{j=1}^n\frac{g(j)}{j}P_{n-j}^g(x)
    \end{align*}
    is valid for every natural number $n$ (see \cite{HN5}).
\end{lm}

Before we proceed to the main part of the work, let us exhibit a special family of polynomials, which plays an important role through the paper.

For $\ell\geq0$, we define the sequence
\begin{align}\label{def: poly n^l}
\psi_\ell(n):=n^\ell
\end{align}
with assigned family of polynomials of the form $P_0^{\psi_\ell}(x):=1$ and

\begin{equation*}
P_n^{\psi_\ell}(x) := \frac{x}{n} \sum_{k=1}^n k^\ell P_{n-k}^{\psi_\ell}(x)\hspace{0.5cm}\text{for }n\geq1.
\end{equation*}

\subsection{Polynomials $P_n^{\psi_0}(x)$}
Here, let us also observe that for the sequence $\psi_0(n)=1$, the polynomials $P_n^{\psi_0}(x)$ might be expressed as
\begin{equation}\label{Formula for g_1(n)}
    P_n^{\psi_0}(x)=\frac{1}{n!}\prod_{i=0}^{n-1}(x+i)=\binom{x+n-1}{n}
\end{equation}
for every non-negative integer $n$, which is a direct consequence of the recurrence formula for $P_n^{\psi_0}(x)$ and the fact that $
\psi _{0}\left( m\right)
=1$ for each $m\in\mathbb{N}$. Therefore, we get the followings.

\begin{cor}\label{Corollary: g_1(n)}
The inequality 
\begin{equation*}
P_a^{\psi_0}(x)P_b^{\psi_0}(x)\geq P_{a+b}^{\psi_0}(x)
\end{equation*}
is satisfied for all $a,b\in\mathbb{N}$ and $x\geq1$ with
equality only for $x=1$.
\end{cor}

\begin{lm}\label{Lemma: P_n^g>f_{g_1,n}}
    Let $\left(g(n)\right)_{n\in\mathbb{N}}$ be a sequence of positive real numbers such that $g(n)\geq1$ for each $n\in\mathbb{N}$. Then, we have
    \begin{align*}
        P_n^g(x)\geq P_n^{\psi_0}(x)
    \end{align*}
    for every $n\in\mathbb{N}$ and $x>0$ with the equality only for $g\equiv \psi_0$.
\end{lm}

\subsection{Polynomials $P_n^{\psi_1}(x)$}

Here, we focus on the special case $P_n^{\psi_1}(x)$ that corresponds to the well-known Laguerre polynomials. Let us recall that $\alpha$-associated Laguerre polynomial $L_n^{(\alpha)}(x)$ is defined by
$$L_n^{(\alpha)}(x)=\sum_{k=0}^n\binom{n+\alpha}{n-k}\frac{(-x)^{k}}{k!}.$$
These polynomials have the leading coefficients $(-1)^n/n!$. Furthermore, they fulfill the recurrence relation of the form
$$nL_n^{(\alpha)}(x)=(2n+\alpha-1-x)L_{n-1}^{(\alpha)}(x)-(n+\alpha-1)L_{n-2}^{(\alpha)}(x).$$
It turns out that they are explicitly related to the polynomials considered in this subsection via  
\begin{align*}
P_n^{\psi_1}(x)=\frac{x}{n}\sum_{j=1}^njP_{n-j}^{\psi_1}(x)=\sum_{k=1}^n\frac{x^k}{k!}\binom{n-1}{k-1}=
\frac{x}{n}L_{n-1}^{(1)}(-x),
\end{align*}
where the penultimate equality is a well-known property of Laguerre polynomials. For more details concerning orthogonal polynomials, we refer the reader to Chihara's book \cite{Chihara}.

Now, we are ready to proceed to the main part of the manuscript.

\section{The proof of Theorem \ref{thm:main}}

The main aim of this section is to prove Theorem \ref{thm:main}. In order to do that, we need some auxiliary results. The strategy is to prove an appropriate criterion, and then apply it to conclude the main theorem.
Hence, we start by showing the following.

\begin{thm}\label{Theorem: General B-O (2)}
    Let $(g(n))_{n\in\mathbb{N}}$ be a sequence of positive real numbers with $g(1)=1$, and let $x_0$ be such that the inequalities
    \begin{align*}
        P_{m}^g(x_0)\geq\max_{1\leq i\leq m}\frac{g(m+i)}{g(i)}
    \end{align*}
    hold for every $m\in\mathbb{N}$. Then, we have that
    \begin{align*}
        P_{a}^g(x)P_{b}^g(x)\geq P_{a+b}^g(x)
    \end{align*}
    is satisfied for every $x\geq x_0$ and $a,b\geq1$.
\end{thm}

In order to prove Theorem \ref{Theorem: General B-O (2)}, we need two auxiliary lemmas.

\begin{lm}\label{Lemma: g(n) (1)}
    Let $x_0$ be such that the inequalities
    \begin{align*}
        P_m^{g}(x_0)\geq\max_{1\leq i\leq m}\frac{g(m+i)}{g(i)}.
    \end{align*}
hold for every $m\geq 1$. Then, the inequality
    \begin{align*}
        P_a^{g}(x_0)P_b^{g}(x_0)\geq P_{a+b}^{g}(x_0)
    \end{align*}
   holds for all $a,b\geq1$.
\end{lm}
\begin{proof}
    The proof goes by induction on $n=a+b$. If $n=2$, then we have that
\begin{align*}
P_1^g(x_0)=x_0\geq\frac{g(2)}{g(1)}=g(2)\hspace{0.2cm}\text{and}\hspace{0.2cm}P_1^g(x_0)^2-P_2^g(x_0)=\frac{x_0}{2}\left(x_0-g(2)\right).
\end{align*}    
    
Thus, let us assume that $n=a+b\geq3$ and $a\geq b$. The induction hypothesis and the recursion formulae for the considered polynomials assert that
    \begin{align*}
        &P_a^{g}(x_0)P_b^{g}(x_0)-P_{a+b}^{g}(x_0)\\
        =&\frac{x_0}{a}\sum_{j=1}^ag(j)P_{a-j}^g(x_0)P_b^g(x_0)-\frac{x_0}{a+b}\sum_{j=1}^{a+b}g(j)P_{a+b-j}^g(x_0)\\
        \geq&\frac{x_0}{a}\sum_{j=1}^ag(j)P_{a-j}^g(x_0)P_b^g(x_0)-\frac{x_0}{a+b}\sum_{j=1}^{a}g(j)P_{a-j}^g(x_0)P_b^g(x_0)
        \\
        &\phantom{\frac{x_0}{a}\sum_{j=1}^ag(j)P_{a-j}^g(x_0)P_b^g(x_0)}-\frac{x_0}{a+b}\sum_{j=1}^{b}g(a+j)P_{b-j}^g(x_0)\\
        =&\frac{x_0}{a(a+b)}\left(b\sum_{j=1}^ag(j)P_{a-j}^g(x_0)P_{b}^g(x_0)-a\sum_{j=1}^bg(a+j)P_{b-j}^g(x_0)\right)\\
        =&\frac{1
        }{a+b}\left(
        bP_{a}^g(x_0)P_{b}^g(x_0)-x_{0}
        \sum_{j=1}^bg(a+j)P_{b-j}^g(x_0)\right)\\
        =&\frac{x_0}{a+b}\left(P_{a}^g(x_0)\sum_{j=1}^bg(j)P_{b-j}^g(x_0)-\sum_{j=1}^bg(a+j)P_{b-j}^g(x_0)\right)\\
        =&\frac{x_0}{a+b}\sum_{j=1}^b\left(g(j)P_{a}^g(x_0)-g(a+j)\right)P_{b-j}^g(x_0).
    \end{align*}
    Finally, the assumption from the statement together with the law of mathematical induction completes the proof.
\end{proof}

To deduce Theorem \ref{Theorem: General B-O (2)}, we also need the following property.

\begin{lm}\label{Lemma: g(n) 2}
    Under the same assumptions as in Lemma \ref{Lemma: g(n) (1)},  the inequality
    \begin{align*}
        P_a^{g}(x)P_b^{g}(x)
        \geq P_{a+b}^{g}(x)
    \end{align*}
   holds for all $a,b\geq1$ and $x\geq x_0$.
\end{lm}

\begin{proof}
Once again we proceed by induction on $n=a+b$. For $n=2$, the required property follows from the proof of Lemma \ref{Lemma: g(n) (1)}. If $n=a+b\geq3$, then Lemma \ref{Lemma: g(n) (1)} boils down
to the problem
showing that $\frac{\mathrm{d
}}{\mathrm{d}x}\left(P_{a}^g(x)P_{b}^g(x)-P_{a+b}^g(x)\right)>0$ for $x\geq x_0$. Let us assume that $a\geq b$. Lemma \ref{Lemma: Recurrence for P_n^g} and the induction hypothesis ensure that 
    \begin{align*}
        &\frac{\mathrm{d}
        }{\mathrm{d}x}\left(P_{a}^g(x)P_{b}^g(x)-P_{a+b}^g(x)\right)\\
        =&\sum_{j=1}^a\frac{g(j)}{j}P_{a-j}^g(x)P_{b}^g(x)+\sum_{j=1}^b\frac{g(j)}{j}P_{a}^g(x)P_{b-j}^g(x)-\sum_{j=1}^{a+b}\frac{g(j)}{j}P_{a+b-j}^g(x)\\
        \geq&\sum_{j=1}^a\frac{g(j)}{j}P_{a-j}^g(x)P_{b}^g(x)+\sum_{j=1}^b\frac{g(j)}{j}P_{a}^g(x)P_{b-j}^g(x)\\
        &{}-\sum_{j=1}^{a}\frac{g(j)}{j}P_{a-j}^g(x)P_{b}^g(x)-\sum_{j=1}^b\frac{g(a+j)}{a+j}P_{b-j}^g(x)\\
        =&\sum_{j=1}^b\left(\frac{g(j)}{j}P_{a}^g(x)-\frac{g(a+j)}{a+j}\right)P_{b-j}^g(x).
    \end{align*}
    Hence, we obtain that
    \begin{align*}
        &\frac{\mathrm{d}
        }{\mathrm{d}x}\left(P_{a}^g(x)P_{b}^g(x)-P_{a+b}^g(x)\right)\geq\sum_{j=1}^b\left(\frac{g(j)}{j}P_{a}^g(x)-\frac{g(a+j)}{a+j}\right)P_{b-j}^g(x).
    \end{align*}
    Finally, the assumptions from the statement end the proof.
\end{proof}
Combining both Lemma \ref{Lemma: g(n) (1)} and Lemma \ref{Lemma: g(n) 2} ends the proof of Theorem \ref{Theorem: General B-O (2)}.

Now, we are in the position to deal with the main goal of this section. Thanks to our criterion, i.e., Theorem  \ref{Theorem: General B-O (2)}, it suffices to examine the property \eqref{eq:star-n}
for the considered family of arithmetic functions  $g_\ell(n)$ satisfying \eqref{eq:gl-growth}, 
and for all values of $\ell,n\in\mathbb{N}$. In such a setting we have the following estimates for the maximal value \eqref{eq:star-n}.

\begin{lm}\label{lm:max}
Let $g_\ell(n)$ be a normalized sequence satisfying \eqref{eq:gl-growth}. 
Then
\[
\max _{1\leq k\leq n}\frac{g_\ell\left( n+k\right) }{g_\ell\left( k\right) }
\leq \left\{
\begin{array}{ll}
2n, & \ell =0,\\
\left( n+1\right) ^{\ell +1}, & \ell \geq 1.
\end{array}
\right.
\]
\end{lm}

\begin{proof}
By \eqref{eq:gl-growth} we obtain
$\max _{1\leq k\leq n}\frac{g_\ell\left( n+k\right) }{g_\ell\left( k\right) }
\leq \max _{1\leq k\leq n}\left( \frac{n}{k}+1\right) ^{\ell +1}k$.
We pass to the continuous problem for $x\in \left[ 1,n\right] $ and consider
$\max _{1\leq x\leq n}\left( \frac{n}{x}+1\right) ^{\ell +1}x$. By
differentiation we obtain
$\frac{\mathrm{d}}{\mathrm{d}x}\left( \frac{n}{x}+1\right) ^{\ell +1}x
=\left( \ell +1\right) \left( \frac{n}{x}+1\right) ^{\ell }
\left( -\frac{n}{x^{2}}\right) x+\left( \frac{n}{x}+1\right) ^{\ell +1}
={\left( \frac{n}{x}+1\right) ^{\ell }
}\frac{1}{x
}\left(
x
-
\ell
{n}
\right) $.
For $\ell
=0$ this is positive for $x>0$. For $\ell \geq 1$
this is non-positive for $x\leq n\leq
\ell
n$.
Therefore, the maximal value is $2
n$ for $\ell
=0$ and
$\left( n+1\right) ^{\ell +1}$ for $\ell \geq 1$.
\end{proof}

At this point, let us consider the cases when $\ell\geq1$ and $1\leq n \leq 8$. In order to do that, one needs to investigate \eqref{eq:star-n} for all $1\leq n\leq8$ one by one (the case of $n=1$ follows from our general assumption, i.e., $x_0=g_\ell(2)$). Since all of the cases are very similar to each other, we only examine the situations when $2\leq n\leq 4$ leaving the remaining ones as an easy exercise to the interested reader. Recall that $g_\ell(2)\geq2^\ell$. By Lemma \ref{lm:max}, we just need to verify the last inequality in each line of

\begin{align*}
P_2^{g_\ell}(x_0)&=x_0^2\geq 2^{2\ell}\geq3^{\ell+1},\\
P_3^{g_\ell}(x_0)&=\frac{x_0}{3}\left(x_0^2+x_0^2+g_\ell(3)\right)\geq\frac{2\cdot8^\ell+6^\ell}{3}\geq4^{\ell+1},\\
P_4^{g_\ell}(x_0)&\geq\frac{5\cdot16^\ell}{12}+\frac{12^\ell}{3}+\frac{8^\ell}{4}\geq5^{\ell+1}.
\end{align*}

One can check that the required inequality for $P_{n}
^{g_{\ell }}\left( x_{0}\right)
$ ($n\in\{2,3,4\}$) is satisfied for each $\ell\geq6-n$. For $5\leq n\leq 8$, it turns out that the appropriate inequalities are fulfilled for all values of $\ell\geq2$.

Now, we restrict ourselves to the cases when $\ell\geq2$ (the case of $\ell=1$ is special and is considered separately). We can proceed by the induction on $n$ to prove that the inequality $P_n^{g_\ell}(x_0)\geq(n+1)^{\ell+1}$ is satisfied for every $n\geq 4$. The cases of $4\leq n\leq8$ have been already verified. Hence, we assume that $n\geq9$. The induction hypothesis together with our estimates for $g_\ell(k)$ guarantees that
\begin{align*}
P_n^{g_\ell}(x_0)=\frac{x_0}{n}\sum_{k=1}^{n}g_\ell(k)P_{n-k}^{g_\ell}(x_0)>\frac{x_0}{n}\sum_{k=1}^{n-4}k^\ell (n-k+1)^{\ell+1}.
\end{align*}
If we consider the function $x^\ell(n-x+1)^{\ell+1}$, then its derivative $\frac{\mathrm{d}
}{\mathrm{d}x}x^\ell(n-x+1)^{\ell+1}=x^{\ell-1}(n-x+1)^\ell\left(\ell(n+1)-(2\ell+1)x\right)$ is positive for all $0<x<\ell(n+1)/(2\ell+1)$. Since $\ell\geq2$ and $n\geq9$, it follows that $2n/5<\ell(n+1)/(2\ell+1)$ and $n-4>2n/5+1$. Therefore, we obtain that
\begin{align*}
\frac{x_0}{n}\sum_{k=1}^{n-4}k^\ell (n-k+1)^{\ell+1}>\frac{x_0}{n}\frac{2n}{5}n^{\ell+1}\geq\frac{(2n)^{\ell+1}}{5},
\end{align*}
where the last inequality is a consequence of $x_0\geq
2^\ell$. Thus, it suffices to prove that 
\begin{align*}
\frac{(2n)^{\ell+1}}{5}\geq(n+1)^{\ell+1}\hspace{0.5cm}\text{or, equivalently,}\hspace{0.5cm}\left(2-\frac{2}{n+1}\right)^{\ell+1}\geq5.
\end{align*}
If we put $n=9$ (that is the worst possible case), then the last of the inequalities holds for all $\ell\geq2$, which concludes the proof in the case of $\ell\geq2$.

Next, we pass to the separate case, i.e., $\ell=1$. For $n\geq1$, we have that
\begin{align*}
P_n^{g_1}(x)
\geq\sum_{j=1}^n\frac{x^k}{k!}\binom{n-1}{k-1}=
\frac{x}{n}L_{n-1}^{(1)}(-x),
\end{align*}
which is a consequence of 
$g_{1}\left( n\right) \geq \psi _{1}\left( n\right) $. Since we have already dealt with the cases when $n\leq8$, let us require that $n\geq9$ and set
\begin{align*}
t(n):=\sum_{k
=1}^5\frac{2^k}{k!}\binom{n-1}{k-1}-(n+1)^2=\frac{n^4}{90}-\frac{11n^2}{18}-\frac{4 n}{3}-\frac{1}{15}\leq P_n^{g_1}(x_0)-(n+1)^2.
\end{align*}
It is transparent that the third derivative $\frac{
\mathrm{d}^3}{\mathrm{d}
n^{3}}t(n)$ is positive for all values of $n>0$. Moreover, $\frac{
\mathrm{d}^2}{\mathrm{d}
n^{2}}t(9)>0$, $\frac{\mathrm{d}
}{\mathrm{d}n}t(9)>0$ and $t(9)>0$, which implies that for $\ell=1$ the difference $P_n^{g_1}(x_0)-(n+1)^2$ is positive for each $n\geq9$, and for every $n\leq8$ we need to require some extra conditions that could guarantee the inequality
\begin{align*}
P_n^{g_1}(x_0)\geq\max_{1\leq k\leq n}\frac{g_1(n+k)}{g_1(k)}.
\end{align*}
This completes the proof of Theorem \ref{thm:main}.

\section{Applications}

\subsection{Sequences related to the partition function}

In this subsection, we use Theorem \ref{thm:main} in practice to prove two generalizations of the submultiplicative property of D'Arcais polynomials \cite{HN6}. The first of them is related to polynomials $P_n^{\sigma_d}(x)$ with $\sigma_d(n):=\sum_{j\mid
n}j^d$. Note that $
P_n^{\sigma}(x)$ corresponds to the D'Arcais polynomials arising from the partition function. On the other hand, it turns out that $P_n^{\sigma_2}(x)$ are polynomials associated
with the plane partition function $pp(n):=P_n^{\sigma_2}(1)$. Unfortunately, there is no
known combinatorial interpretation of $P_n^{\sigma_d}(1)$ for $d\geq3$. Because of that, we decide to apply Theorem \ref{thm:main} to
another family of polynomials. Let $\lambda_{\ell}(n)$ be defined recursively by $\lambda_0(n):=1$ and
\begin{equation*}
\lambda_\ell(n):=\sum_{j|n}j\lambda_{\ell-1}(j) \hspace{0.2cm}\text{for }\ell\geq1.
\end{equation*}
In such a setting, we have that $P_{n}^{\sigma }
\left( x\right)
=P_n^{\lambda_1}(x)$, and $P_n^{\lambda_{\ell}}(x)$ possesses an interesting combinatorial interpretation for each $\ell\geq0$. More accurately, let us put
\begin{align*}
    C_{\ell,n}:=\left\{\left(\pi_1,\pi_2,\ldots,\pi_\ell\right)\in S_n^\ell:\pi_i\pi_j=\pi_j\pi_i\text{ for every }1\leq i,j\leq \ell\right\},
\end{align*}
where $S_n$ denotes the symmetric group. In other words, $\#C_{\ell,n}$ is the number of commuting $\ell$-tuples in $S_n$. Bryan and Fulman \cite{BF} proved that the equations
\begin{align*}
    \sum_{n=0}^\infty P_n^{\lambda_\ell}(x)q^n=\prod_{n=1}^\infty\frac{1}{(1-q^n)^{x\lambda_{\ell-1}(n)}}=\exp{\left(x\sum_{n=1}^\infty \lambda_{\ell}(n)\frac{q^n}{n}\right)}
\end{align*}
are valid for every positive integer $\ell\geq1$, and $P_n^{\lambda_\ell}(1)=\#C_{\ell,n}/n!$. For more information concerning the set of commuting permutations $C_{\ell,n}$ we encourage the reader to see \cite{ABD, AHN}.

\begin{re}\label{remark:p_A}{\rm
It is worth noting that the values of $P_n^{\lambda_\ell}(k)$ for $k\in\mathbb{N}$ have an interesting combinatorial interpretation in the context of the so-called $A$-partition functions. Let us recall that for a fixed multiset $A$ of positive integers such that each number $j$ occurs only finitely many times in $A$, an $A$-partition is a partition with parts in $A$. Moreover, two $A$-partitions of a given $n$ are considered the same if they differ only in the order of their parts. The total number of $A$-partitions of is denoted by $p_A(n)$. We have that $P_n^{\lambda_\ell}(k)=p_{A_k}(n)$ for $k\in\mathbb{N}$, where $A_k$ is such a multiset that each number $j$ occurs exactly $k\lambda_{\ell-1}(j)$ times in $A_k$.
}\end{re}

Now, we are in the position to exhibit the submultiplicative properties of $P_n^{\sigma_d}(x)$ and $P_n^{\lambda_\ell}(x)$, respectively.

\begin{cor}\label{Corollary: B-O (1) for sigma}
    Let $d\geq1$. The inequality
    \begin{align*}
        P_a^{\sigma_d}(x)P_b^{\sigma_d}(x)\geq P_{a+b}^{\sigma_d}(x)
    \end{align*}
    holds true for all natural numbers $a$ and $b$, and $x\geq 2^{d}+1$.
\end{cor}

\begin{cor}\label{Corollary: B-O (2) for gl}
    Let $\ell\geq1$. The inequality
    \begin{align*}
        P_a^{\lambda_\ell}(x)P_b^{\lambda_\ell}(x)\geq P_{a+b}^{\lambda_\ell}(x)
    \end{align*}
    holds true for all natural numbers $a$ and $b$, and $x\geq 2^{\ell+1}-1$.
\end{cor}

Both of the corollaries follow from Theorem \ref{thm:main}. However, to apply that result, we first need to observe that $\sigma(n)=\lambda_{1}(n)$,
\begin{align*}
    n^d\leq\sigma_d(n)=n^d\sum_{j\mid
n}\frac{1}{j^d}\leq n^d+n^d\int_1^n\frac{\mathrm{d}x}{x^d}=\frac{n}{d-1}\left(dn^{d-1}-1\right)<n^{d+1},
\end{align*}
and
\begin{align*}
    n^{\ell}\leq \lambda_{\ell}(n)\leq n^{\ell+1}
\end{align*}
hold for $\ell,n\in \mathbb{N}$ and $d\in\mathbb{N}_{\geq2}$, where the last of the inequalities can be found in \cite{ABD}. Thus, in order to conclude both of the corollaries it suffices to check the additional conditions for $d,\ell\in\{1,2,3\}$, what can be done one by one.

\subsection{The case of $\psi_\ell$}

Here, we investigate the submultiplicative property for the family of polynomials $P_n^{\psi_\ell}(x)$ associated
with the sequences $\psi_\ell(n)=n^\ell$ (introduced in Section 2). Thanks to Corollary \ref{Corollary: g_1(n)} and Theorem \ref{thm:main}, it suffices to examine \eqref{eq:star-n} for $\ell=1$ (and $1\leq n\leq 8$), $\ell=2$ (and $1\leq n\leq 3$), and $\ell=3$ (and $1\leq n\leq 2$), what can be done one by one. This leads to the following.
\begin{cor}
Let $\ell\geq0$. The inequality
    \begin{align*}
        P_a^{\psi_\ell}(x)P_b^{\psi_\ell}(x)\geq P_{a+b}^{\psi_\ell}(x)
    \end{align*}
    holds true for all natural numbers $a$ and $b$, and $x\geq 2^{\ell}$.
\end{cor}

\subsection{The overpartition function}

As the last example we examine the submultiplicative property for the polynomization of the overpartition function due to Li \cite{Li}. More precisely, we consider the family of polynomials $P_{n}^{\bar{\sigma}}(x)$, where $\bar{\sigma}(2^ml):=2^{m+1}\sigma(l)$ for any non-negative integer $m$ and an odd number $l\geq1$.

At
first glance, it seems that we can not simply use neither Theorem \ref{thm:main} nor Theorem \ref{Theorem: General B-O (2)} to obtain the required phenomenon as $\bar{\sigma}(1)=2$. However, it turns out that we may consider a small modification of the function $\bar{\sigma}(n)$. Let us just set $\hat{\sigma}(n):=\bar{\sigma}(n)/2.$ In such a setting one can easily verify that $P_n^{\bar{\sigma}}(x)=P_n^{\hat{\sigma}}(2x)$ for every non-negative integer $n$ and $x\in\mathbb{R}$. Also it is clear that $\hat{\sigma}(1)=1$. Moreover, for every natural number $n=2^ml$ (where $m$ and $l$ are as before) we have that
\begin{equation*}\label{bounds hat:sigma}
n\leq\hat{\sigma}(n)=\hat{\sigma}(2^ml)=2^m\sigma(l)\leq 2^ml(1+\ln{l})\leq2^ml^2\leq n^2.
\end{equation*}
Because of this, we can apply Theorem \ref{thm:main} with $x_0:=2$, and deduce the following.

\begin{thm}[Li]\ \\
The inequality
\begin{equation*}
P_{a}^{\bar{\sigma}}(x)P_{b}^{\bar{\sigma}}(x)\geq P_{a+b}^{\bar{\sigma}}(x)
\end{equation*}
holds true for all natural numbers $a$ and $b$, and $x\geq1$. 
\end{thm}

\section{Proof of Theorem \ref{thm:main 2}}

Before we pass to the main part of this section, let us exhibit some motivation behind Theorem \ref{thm:main 2}.

\begin{re}{\rm It turns out that sequences satisfying $1\leq g(n)\leq \sigma(n)$ covers all possible $A$-partition functions (mentioned in Remark \ref{remark:p_A}), where $A$ is a subset of natural numbers. More precisely, any $A$-partition function $p_A(n)$ fulfills  the recurrence relation of the form
\begin{align*}
p_A(n)=\frac{1}{n}\sum_{j=1}^n\sigma_A(j)p_A(n-j),
\end{align*}
where $\sigma_A(k)$ is defined as a sum of those divisors of $k$ that belong to the set $A$. In particular, it implies that whenever $A$ is a subset of the set of natural numbers with $1\in A$ we have that
\begin{align*}
1\leq\sigma_A(n)\leq\sigma(n)\leq n(1+\ln(n))
\end{align*}
for every positive integer $n$. Moreover, the polynomials associated
with the sequence $\sigma_A(n)$ are directly connected to $p_A(n)$ via the equality $p_A(n)=P_n^{\sigma_A}(1)$.
}\end{re}

Now we proceed to the main part of this section. Since the considered sequence $g(n)$ is bounded from below by $1$ for each $n$, we can simply estimate the associated polynomials from below by the corresponding polynomials from \eqref{Formula for g_1(n)} for every $x\geq
1$. For the upper bound, we utilize the classic estimation $\sigma(n)\leq n(1+\ln(n))$, which maintains that
\begin{align*}
\max_{1\leq k\leq n}\frac{g(n+k)}{g(k)}\leq
\max _{1\leq k\leq n}g\left( n+k\right)
\leq \max _{1\leq k\leq n}\sigma \left(
n+k\right) \leq 2n(1+\ln(2n)).
\end{align*}
Therefore, if we optimally take  $x_0=3$, then it is enough to verify the inequality
\begin{align*}
\binom{n+2}{2}\geq2n(1+\ln(2n)),
\end{align*} 
which holds for every $n\geq15$. For the smaller values of $n$, we check the condition
\begin{align*}
\binom{n+2}{2}\geq \max _{1\leq k\leq n}\sigma \left(
n+k\right)
\end{align*}
one by one. It turns out that it is fulfilled for all $1\leq n\leq 15$
except for $n=2$ and $n=3$. However, since
\begin{align*}
P_{2}^{g}\left( 3\right) =&\frac{3}{2}\left(
3+g\left( 2\right) \right) \geq 6>4\geq g\left( 3\right)
,\\
P_{3}^{g}\left( 3\right)
=&\frac{9}{2}+\frac{9}{2}g\left( 2\right) +g\left( 3\right)
\geq 10>7\geq \max \left\{ g\left( 4\right) ,
\frac{g\left( 5\right) }{g\left( 2\right) }\right\} ,
\end{align*}
if we additionally require that
\begin{align*}
P_2^g(3)
\geq
\frac{g(4)}{g(2)}
\hspace{0.5cm}
\text{and}\hspace{0.5cm}
P_3^g(3)
\geq
\frac{g(6)}{g(3)}
,
\end{align*}
or, equivalently,
\begin{align*}
3g(2)\left(g(2)+3\right)\geq2g(4)\hspace{0.5cm}\text{and}\hspace{0.5cm}g(3)\left(2g(3)+9g(2)+9\right)\geq2g(6),
\end{align*}
then we can deduce that the submultiplicative property is satisfied for the considered family of polynomials $P_n^g(x)$ for all $x\geq 3$ because of  Theorem \ref{Theorem: General B-O (2)}; otherwise we 
require that $x_0=4$ and verify
that the inequality
\begin{align*}
\binom{n+3}{n}\geq2n(1+\ln(2n))
\end{align*}
holds true for all $n\geq1$, which is an elementary exercise.

In conclusion, we obtain Theorem \ref{thm:main 2}.

\section*{Acknowledgments}
The first author was supported by the National Science Center grant no.\ 2024/53/N/ST1/01538.

The last author was partially supported by the Faculty of Applied Mathematics AGH UST statutory tasks within subsidy of the Ministry of Science and Higher Education (Poland).

\end{document}